\documentclass[10pt,a4paper,oneside,english]{amsart}
\usepackage[T1]{fontenc}
\usepackage[latin9]{inputenc}
\usepackage{units}
\usepackage{amsthm}
\usepackage{amssymb}

\makeatletter

\numberwithin{equation}{section}
\numberwithin{figure}{section}
\theoremstyle{plain}
\newtheorem{thm}{\protect\theoremname}[section]
  \theoremstyle{plain}
  \newtheorem{lem}[thm]{\protect\lemmaname}
  \theoremstyle{plain}
  
  \theoremstyle{remark}
  \newtheorem{rem}[thm]{\protect\remarkname}
  \theoremstyle{plain}
  
  \newtheorem{fact}[thm]{Fact}
\@ifundefined{date}{}{\date{}}

\usepackage{amsthm}

\usepackage{cite}

\newtheorem{theorem}{Theorem}[section]

\newtheorem{dfn}[theorem]{Definition}
\newtheorem{exa}[theorem]{Example}

\providecommand{\lemmaname}{Lemma}
  \providecommand{\propositionname}{Proposition}
  \providecommand{\remarkname}{Remark}
\providecommand{\theoremname}{Theorem}

\usepackage{babel}
\providecommand{\lemmaname}{Lemma}
  \providecommand{\propositionname}{Proposition}
  \providecommand{\remarkname}{Remark}
\providecommand{\theoremname}{Theorem}

\usepackage{babel}
\providecommand{\corollaryname}{Corollary}
  \providecommand{\lemmaname}{Lemma}
  \providecommand{\propositionname}{Proposition}
  \providecommand{\remarkname}{Remark}
\providecommand{\theoremname}{Theorem}

\makeatother

\usepackage{babel}
  \providecommand{\corollaryname}{Corollary}
  \providecommand{\lemmaname}{Lemma}
  \providecommand{\propositionname}{Proposition}
  \providecommand{\remarkname}{Remark}
\providecommand{\theoremname}{Theorem}

\begin{document}

\title [Weak$^*$-fpp  in $\ell_1$ and Polyhedrality in Lindenstrauss Spaces] {Weak$^*$ Fixed Point Property  in $\ell_1$ and Polyhedrality in Lindenstrauss Spaces}

\author[E. Casini]{Emanuele Casini}
\address{\textsc{Emanuele Casini}, \rm{Dipartimento di Scienza e Alta Tecnologia, Universit\`{a} dell'Insubria,
via Valleggio 11, 22100 Como, Italy.}}
\email{emanuele.casini@uninsunbria.it}

\author[E. Miglierina]{Enrico Miglierina}
\address{\textsc{Enrico Miglierina}, \rm{Dipartimento di Discipline Matematiche, Finanza Matematica ed Econometria,
Universit\`{a} Cattolica del Sacro Cuore, Via Necchi 9, 20123 Milano,
Italy.}}
\email{enrico.miglierina@unicatt.it}

\author[\L. Piasecki]{\L ukasz Piasecki}
\address{\textsc{\L ukasz Piasecki}, \rm{Instytut Matematyki,
Uniwersytet Marii Curie-Sk{\l}odowskiej, Pl. Marii Curie-Sk{\l}odowskiej 1, 20-031 Lublin, Poland.}}
\email{piasecki@hektor.umcs.lublin.pl}

\author[R. Popescu]{Roxana Popescu}
\address{\textsc{Roxana Popescu}, \rm{Department of Mathematics, University of Pittsburgh, Pittsburgh, PA 15260,
	USA.}}
\email{rop42@pitt.edu}

\thanks{\textbf{Acknowledgements}: The authors thank Stanis{\l}aw Prus for helpful conversations on and around the
article \cite{Japon-Prus2004}. They also thank Libor Vesel\'{y} for useful discussions.}

\numberwithin{equation}{section}

\begin{abstract}
The aim of this paper is to study the $w^*$-fixed point property for nonexpansive mappings in the duals of separable Lindenstrauss spaces by means of suitable geometrical properties of the dual ball. First we show that a property concerning the behaviour of a class of $w^*$-closed subsets of the dual sphere is equivalent to the   $w^*$-fixed point property. Then,
the main result of our paper shows an equivalence between another, stronger geometrical property of the dual ball and the
stable $w^*$-fixed point property. The last geometrical notion was introduced by Fonf and Vesel\'{y} as a strengthening of the notion of polyhedrality.
In the last section we show that also the first geometrical assumption that we have introduced can be related to a polyhedral concept for the predual space. Indeed, we give a hierarchical structure among various polyhedrality notions in the framework of Lindenstrauss spaces. Finally, as a by-product, we obtain an improvement of an old result about the norm-preserving compact extension of compact operators.
\end{abstract}



\maketitle

\section{Introduction and Preliminaries}
Let $X$ be an infinite dimensional real Banach space and let us denote by $B_X$ its closed unit
ball and by $S_X$ its unit sphere.
We say that a Banach space $X$ is a Lindenstrauss space if its dual is a space $L_1(\mu)$ for some measure $\mu$. 
A nonempty bounded closed and convex subset $C$ of $X$ has the fixed point property
(shortly, fpp) if each nonexpansive mapping (i.e., the mapping $T:C\rightarrow C$ such that $\|T(x)-T(y)\|\leq \|x-y\|$ for all $x,y\in C$) has a fixed point.
A dual space $X^*$ is said to have the $\sigma(X^*,X)$-fixed
point property ($\sigma(X^*,X)$-fpp) if every nonempty, convex,
$w^*$-compact subset $C$ of $X^*$ has the fpp. 

The study of the
$\sigma(X^*,X)$-fpp reveals to be of special interest whenever a
dual space has different preduals. 
For
instance, this situation occurs when we consider the space $\ell_1$
and its preduals $c_0$ and $c$ where it is well-known (see
\cite{Karlovitz1976}) that $\ell_1$ has the $\sigma(\ell_1,c_0)$-fpp
whereas it lacks the $\sigma(\ell_1,c)$-fpp.
The first result of our paper is devoted to a geometrical characterization of the preduals $X$ of $\ell_1$ that induce on $\ell_1$ itself the $\sigma(\ell_1,X)$-fpp. This theorem can be seen as an extension of the characterization given in Theorem 8 in \cite{Japon-Prus2004} and it is based on the studies carried out in \cite{Casini-Miglierina-Piasecki2015}.
However, the main purpose of the present paper is to investigate the stability of $\sigma(\ell_1,X)$-fpp.
Generally speaking, stability of fixed point property deals with the following question: let us suppose that a Banach space $X$ has the fixed point property and $Y$ is a Banach space isomorphic to $X$ with "small" Banach-Mazur distance, does $Y$ have fixed point property? This problem has been widely studied for fpp and only occasionally for $w^*$-topology (see \cite{Soardi,Dominguez-Garcia-Japon1998}).
It is worth pointing out that the stability property of $w^*$-fpp previously studied in the literature considers the renormings of $X^*$ whereas it maintains the original $w^*$-topology on the renormed space. Since a renorming of a given Banach space $X^*$ not necessarily is a dual space, we prefer to introduce a more suitable notion of stability for $w^*$-fpp that takes into account each dual space with the proper $w^*$-topology induced by its predual (see Definition \ref{Def Stable fpp}).
In Section \ref{Sec:Stablefpp} we prove that the stability of  $\sigma(\ell_1,X)$-fpp is equivalent to a suitable property concerning the behaviour of $w^*$-limit points of the set of extreme points of $B_{\ell_1}$. 
It is worth to mention that the property playing a key role in Section \ref{Sec:Stablefpp} was already introduced in \cite{Fonf-Vesely2004} by Fonf and Vesel\'{y}, in the completely different setting of polyhedral spaces theory.


The concept of polyhedrality, originally introduced by Klee in \cite{Klee1960}, is widely studied (for detailed survey about various definitions of polyhedrality for infinite dimensional spaces see \cite{Durier-Papini1993,Fonf-Vesely2004}) and it gives a deep insight of geometrical properties of Banach spaces. Beyond its intrinsic interest, polyhedrality has some important applications. For instance, in the framework of Lindenstrauss spaces, it is related to the existence of norm-preserving compact extension of compact operators (see \cite{Casini-Miglierina-Piasecki-Vesely2016} and \cite{Fonf-Lindenstrauss-Phelps2001}). In the last section of the paper we compare the geometrical assumptions used to study $\sigma(\ell_1,X)$-fpp with the main generalizations of polyhedrality already considered in the literature. 
Indeed, the assumption characterizing the preduals of $\ell_1$ satisfying stable $w^*$-fpp is listed as a generalization of polyhedrality in \cite{Fonf-Vesely2004} and we show how the property ensuring the validity of $\sigma (\ell_1,X)$-fpp fits very well in a list of several properties related to the original  definition of polyhedral space. Our results, in the framework of Lindenstrauss spaces, prove that the notions of polyhedrality play an important role also in fixed point theory. We also give a hierarchical structure among various notions of polyhedrality by restricting our attention to Lindenstrauss spaces. These results allow us to prove a new version of an old uncorrect result (Theorem 3 in \cite{Lazar1969})  concerning the norm-preserving compact extension of compact operators.
    
Finally, we collect some notations. If $A\subset X$, then we denote by $\overline{A}$, $\mathrm{conv}(A)$ and $\mathrm{ext}\,A$ the norm closure of $A$, the convex hull of $A$ and the set of the extreme points of $A$  respectively. Moreover, whenever $A \subset X^*$, we denote by $\overline{A}^*$ the $w^*$-closure of $A$.
For $x\in S_X$, we call $D(x)$ the image of $x$ by the duality mapping, i.e., 
$$
D(x)=\left\lbrace x^*\in S_{X^*}:x^*(x)=1\right\rbrace. 
$$


\section{A characterization of $w^*$-fixed point property in $\ell_1$}

The aim of this brief section is to characterize the separable Lindenstrauss spaces $X$ such that $X^*$ has the $\sigma(X^*,X)$-fpp. The present result adds a new characterization of $\sigma (X^*,X)$-fpp to those listed in \cite{Casini-Miglierina-Piasecki2015}. Moreover, the theorem sheds some new light on the relationships between $\sigma (X^*,X)$-fpp and a geometrical feature of the sphere in $X^*$. We will show in Section \ref{sec:poly} that this feature can be interpreted as a polyhedrality requirement on $X$.
\begin{thm}\label{Theorem fpp}
Let $X$ be a separable Lindenstrauss space. Then the following are equivalent.
\begin{enumerate}
  \item[(i)] $X^*$ has the $\sigma(X^*,X)$-fpp;
  \item[(ii)] there is no infinite set $C \subset \mathrm{ext} \, B_{X^*}$ such that $\overline{\mathrm{conv} (C)}^* \subset S_{X^*}$.
\end{enumerate}
\end{thm}
\begin{proof}
Let us suppose that $X$ is a separable Lindenstrauss space with nonseparable dual. Then by Corollary 3.4 in \cite{Casini-Miglierina-Piasecki2015} $X$ fails the $w^*$-fpp and it also fails property (ii) by Theorem 2.3 in \cite{Lazar-Lindenstrauss1971} and Theorem 2.1 in \cite{Casini-Miglierina-Piasecki-Vesely2016}. Therefore we limit ourselves to consider the case where $X^*$ is isometric to $\ell_1$. Let us suppose that $X^*$ fails the $w^*$-fixed point property. Then, by Theorem 4.1 in \cite{Casini-Miglierina-Piasecki2015}, there exist a subsequence $\left\lbrace e^*_{n_k}\right\rbrace $ of the standard basis of $\ell_1$ and a point $e^*\in S_{\ell_1}$ such that $e^*_{n_k}\xrightarrow{\sigma(\ell_1,X)}e^*$ and $e^*(n_k)\geq 0$. By taking $C=\left\lbrace e^*_{n_k}\right\rbrace$, it is easy to see that $X$ does not satisfy property (ii).
On the other hand, without loss of generality we can assume that there exists a set $C=\left\lbrace e^*_{n_k}\right\rbrace$ where $\left\lbrace e^*_{n_k}\right\rbrace$ is a $w^*$-convergent subsequence of the standard basis of $\ell_1$. Let us denote by $e^*$ the $w^*$-limit of $\left\lbrace e^*_{n_k}\right\rbrace$ then, by recalling Corollary 2 in \cite{Japon-Prus2004}, we have 
\[
\overline{\mathrm{conv}(C)}^*=\overline{\mathrm{conv} \, \left(\left\lbrace e^*, e^*_{n_1},e^*_{n_2}, \cdots \right\rbrace  \right) } \subset S_{\ell_1}.
\]
By adapting to our setting the method developed in the last part of the proof of Theorem 8 in \cite{Japon-Prus2004}, we easily find a nonexpansive mapping fixed point free from $\overline{\mathrm{conv}(C)}^*$ to $\overline{\mathrm{conv}(C)}^*$. 
\end{proof}


\section{A characterization of stable $w^*$-fixed point property in $\ell_1$}\label{Sec:Stablefpp}

This section is devoted to the main result of the paper. We characterize the stable $w^*$-fpp for the duals of separable Lindenstrauss spaces by means of a suitable property describing the interplay between $w^*$-topology and the geometry of the dual ball $B_{X^*}$.
To our knowledge, the first result dealing with the stability of $w^*$-fpp for this class of spaces is the following:

\begin{thm}(\cite{Soardi})
Let $Y$ be a Banach space such that $d(\ell_1,Y)<2$. Then $Y$ has the $w^*$-fpp.
\end{thm}

We remark that the statement of the previous theorem implicitly assume that $\ell_1$ is endowed with the $\sigma(\ell_1,c_0)$-topology. Moreover, since $Y$ is not necessarily a dual space, the author of \cite{Soardi} considers the original topology $\sigma(\ell_1,c_0)$ on $Y$. For this reason, this approach does not allow to consider a true $w^*$-fpp in the space $Y$. In order to avoid this undesirable feature we introduce a different definition of stability for the $w^*$-fpp. 

\begin{dfn}\label{Def Stable fpp}
A dual space $X^*$ enjoys the \textit{stable} $\sigma(X^*,X)$-fpp if there exists a real number $\gamma >1$ such that $Y^*$ has the $\sigma(Y^*,Y)$-fpp whenever $d(X,Y)< \gamma$, where $ d(X,Y)$ is the Banach-Mazur distance between $X$ and $Y$.
\end{dfn}

We recall that every nonseparable dual of a separable Lindenstrauss space fails the $w^*$-fpp (see Corollary 3.4 in \cite{Casini-Miglierina-Piasecki2015}). Therefore, in the sequel of this section, we restrict our attention to the preduals of $\ell_1$.
If $A \subset X^*$, then we denote by $A'$ the set of all $w^*$-limit points of $A$:
$$
A'=\left\lbrace x^*\in X^*:x^* \in \overline{(A\setminus \left\lbrace x^*\right\rbrace )}^*\right\rbrace.
$$

The following Lemma shows how the geometrical assumption that will play a crucial role in our characterization of stable $w^*$-fpp influences the behaviour of some sequences in $\ell_1$.

\begin{lem}\label{Lemma sequences}
Let $X$ be a predual of $\ell_1$.
\begin{enumerate}
 \item[(a)] For every sequence $\left\lbrace x_n^*\right\rbrace \subset \ell_1$ coordinatewise converging to $x^*_0$ and such that $\lim_{n\rightarrow \infty}\left\|x^*_n-x^*_0\right\|$ exists, it holds
 \begin{equation*}
 \lim_{m \rightarrow \infty}\lim_{n \rightarrow \infty} \left\|x^*_n-x^*_m\right\|= 2\lim_{n\rightarrow \infty} \left\|x^*_n-x^*_0\right\|.
 \end{equation*}
 \end{enumerate}

 If, in addition, $\left(\mathrm{ext} \, B_{\ell_1}\right)'\subset rB_{\ell_1}$ for some $0\leq r<1$, then
\begin{enumerate}
 \item[(b)] for every sequence $\left\lbrace x_n^*\right\rbrace \subset \ell_1$ such that $\left\lbrace x_n^*\right\rbrace$ is $\sigma(\ell_1,X)$-convergent to $x^*$ and $\left\lbrace x_n^*\right\rbrace$ tends to $0$ coordinatewise, it holds
 \begin{equation*}
 \left\|x^*\right\|\leq r \liminf_{n\rightarrow \infty}\left\|x_n^*\right\|;
 \end{equation*}

 \item[(c)] for every sequence $\left\lbrace x_n^*\right\rbrace \subset \ell_1$ such that $\left\lbrace x_n^*\right\rbrace$ is $\sigma(\ell_1,X)$-convergent to $x^*$, up to a subsequence, it holds
 \begin{equation*}
  \lim_{n\rightarrow \infty}\left\|x^*_n-x^*\right\|\leq\dfrac{1+r}{2}\lim_{m\rightarrow \infty} \lim_{n\rightarrow \infty}\left\|x^*_n-x^*_m\right\|.
  \end{equation*}
\end{enumerate} 
\end{lem}

\begin{proof}
Assertion (a) is a straightforward consequence of the following fact:  for every sequence $\left\lbrace x_n^*\right\rbrace \subset \ell_1$ coordinatewise converging to $0$ and such that $\lim_{n\rightarrow \infty}\left\|x^*_n\right\|$ exists it holds
 \begin{equation*}
 \lim_{n \rightarrow \infty} \left\|x^*_n+x^*\right\|= \lim_{n\rightarrow \infty} \left\|x^*_n\right\|+ \left\|x^*\right\|
 \end{equation*}
 for every $x^* \in \ell_1$.
 
Now we prove assertion (b). For all $m,n \in \mathbb{N}$ and $x\in B_X$, we have
\begin{equation*}
x_n^*(x)= \sum_{i=1}^{m}x_n^*(i) e_i^*(x)+\sum_{i=m+1}^{\infty}
x_n^*(i)e_i^*(x)
\leq\sum_{i=1}^{m}\left|x_n^*(i)\right|+\left\|x_n^*\right\|
\sup_{i\geq{m+1}}\left|e_i^*(x)\right|.
\end{equation*}
Since $\left(\textrm{ext}B_{\ell_1}\right)' \subseteq r B_{\ell_1}$, then $\lim\limits_{m\rightarrow \infty} \sup_{i\geq m+1}\left|e_i^*(x)\right| \leq r$ for every $x\in B_X$.
Therefore, we easily see that
\begin{equation*}
x^*(x)=\lim\limits_{n\rightarrow \infty}x_n^*(x)\leq
\liminf_{n\rightarrow \infty}\left\|x_n^*\right\|
\lim\limits_{m\rightarrow \infty}
\left( \sup_{i\geq{m+1}}\left|e_i^*(x)\right|\right)  \leq r
\liminf_{n\rightarrow \infty}\left\|x_n^*\right\|,
\end{equation*}
which proves the thesis of assertion (b).

In order to prove the last assertion, 
without loss of generality, we may assume that there exists $x_0^*\in \ell_1$ such that $\left\lbrace x_n^*\right\rbrace$ is coordinatewise convergent to $x^*_0$ and that $\lim_{n\rightarrow\infty}\left\|x^*_n-x^*\right\|$ and $\lim_{n\rightarrow\infty}\left\|x^*_n-x^*_0\right\|$ exist. By assertions (b) and (a), we get
\begin{eqnarray*}
\lim_{n\rightarrow\infty}\left\|x^*_n-x^*\right\|&=& \lim_{n\rightarrow\infty}\left\|x^*_n-x^*_0\right\|+\left\|x^*_0-x^*\right\|
\\&\leq& \lim_{n\rightarrow\infty}\left\|x^*_n-x^*_0\right\|+r\lim_{n\rightarrow\infty}\left\|x^*_n-x^*_0\right\|
\\&=& \dfrac{1+r}{2}\lim_{m \rightarrow\infty}\lim_{n\rightarrow\infty}\left\|x^*_n-x^*_m\right\|.
\end{eqnarray*}
\end{proof}

For the sake of convenience of the reader, we recall the following known result.

\begin{lem}\label{lemma Soardi}(\cite{Soardi})
Let $Y^*$ be a dual Banach space, $K\subset Y^*$ be a convex $w^*$-compact subset and $T:K \rightarrow K$ be a nonexpansive mapping. Then, for every $x^* \in K$ there is a closed convex subset $H(x^*)\subset K$ which is invariant under $T$ and satisfies
\begin{enumerate}
 \item[(a)] $\mathrm{diam} \left(H(x^*)\right) \leq \sup_n \left\|x^*-T^n x^*\right\|;$
 \item[(b)] $\sup_{y^* \in H(x^*)}\left\|x^*-y^*\right\| \leq 2\sup_n \left\|x^*-T^n x^*\right\|$.
\end{enumerate}

\end{lem}

\begin{thm}\label{Theorem stable fpp}
Let $X$ be a predual of $\ell_1$. Then the following are equivalent.
\begin{enumerate}
 \item[(i)]$\ell_1$ has the stable $\sigma(\ell_1,X)$-fpp;
 \item[(ii)] $\left(\mathrm{ext} \, B_{\ell_1}\right)'\subset rB_{\ell_1}$ for some $0\leq r<1$.
\end{enumerate}
\end{thm}

\begin{proof}
We first prove that (i) implies (ii).
By contradiction, let us suppose that $X$ does not satisfies property (ii).
For clarity we divide the proof of this implication into three parts.

\textbf{Step 1.} (\textit{Renorming.}) Let $\varepsilon \in (0,1)$ and choose a
subsequence $(e_{n_k}^*)$ of the standard basis $(e_n^*)$ in $\ell_1$
such that $(e_{n_k}^*)$ is $w^*$-convergent to $e^*\neq 0$, and
\begin{equation}\label{estimate 5}
\sum_{k=1}^{\infty}\left|e^*(n_k)\right|<
\frac{\varepsilon}{4}\left\|e^*\right\|.
\end{equation}

Consider the subset $C$ of $\ell_1$ defined by
$$C= \left\{\alpha_0 e^*+\sum_{k=1}^{\infty}\alpha_k e_{n_k}^*: \alpha_i \geq 0 \textrm{ for all }i\in
\mathbb{N}\cup\left\{0\right\} \textrm{ and }
 \sum_{i=0}^{\infty}\alpha_i=1\right\}.$$
It is easy to check that $C$ is convex and $w^*$-compact (see
e.g. Corollary 2 in \cite{Japon-Prus2004}). From (\ref{estimate 5})
it follows that for every $x=\alpha_0
e^*+\sum_{k=1}^{\infty}\alpha_k e_{n_k}^* \in C$ we have

\begin{eqnarray*}
\left\|x\right\|&=&\alpha_0 \sum_{j\in \mathbb{N}\setminus
\left\{n_k\right\}}\left|e^*(j)\right|+
\sum_{j\in\left\{n_k\right\}}\left|\alpha_0 e^*(j)+\alpha_k\right|
\\&\geq&\alpha_0 \sum_{j\in \mathbb{N}\setminus
\left\{n_k\right\}}\left|e^*(j)\right|+\sum_{k=1}^{\infty}\alpha_k-
\alpha_0\sum_{j\in\left\{n_k\right\}}\left|e^*(j)\right|
\\&\geq&\alpha_0 \sum_{j\in \mathbb{N}\setminus
\left\{n_k\right\}}\left|e^*(j)\right|+(1-\alpha_0) \sum_{j\in
\mathbb{N}\setminus \left\{n_k\right\}}\left|e^*(j)\right|-
\alpha_0\sum_{j\in\left\{n_k\right\}}\left|e^*(j)\right|
\\&\geq&\sum_{j\in \mathbb{N}\setminus
\left\{n_k\right\}}\left|e^*(j)\right|-\sum_{j\in\left\{n_k\right\}}\left|e^*(j)\right|
\\&\geq&(1-\frac{\varepsilon}{4})\left\|e^*\right\|-\frac{\varepsilon}{4}\left\|e^*\right\|
\\&=&(1-\frac{\varepsilon}{2})\left\|e^*\right\|.
\end{eqnarray*}
Let
$$K:=\frac{1}{(1-\varepsilon)\left\|e^*\right\|}C.$$
The set $K$ is convex, $w^*$-compact and $K\cap
B_{\ell_1}=\emptyset$. Next we define the set $D \subset \ell_1$ by
$$D=\textrm{conv}(B_{\ell_1}\cup K\cup -K).$$
It is easy to check that $D$ is convex, symmetric, $w^*$-compact,
and $0$ is its interior point. Therefore $D$ is a dual unit ball of an
equivalent norm $\left\|\left|\cdot\right|\right\|$ on $X$. Let
$Y=(X,\left\|\left|\cdot\right|\right\|)$. Obviously, $D=B_{Y^*}$
and
$$B_{\ell_1}\subset D \subset \frac{1}{(1-\varepsilon)\left\|e^*\right\|} B_{\ell_1},$$
so
\begin{eqnarray}\label{estimate 6}
d(X,Y)\leq \frac{1}{(1-\varepsilon)\left\|e^*\right\|}.
\end{eqnarray}

Consider a subspace $Z\subset Y^*$ defined as
$$Z=\overline{\textrm{span}\left( \left\{e^*,e_{n_1}^*,e_{n_2}^*,\dots\right\}\right) }=
\overline{\textrm{span}\left( \left\{e_0^*,e_{n_1}^*,e_{n_2}^*,\dots\right\}\right) },$$
where
$$e_0^*:=\frac{e^*-\sum\limits_{j=1}^{\infty}e^*(n_j)e^*_{n_j}}{\left\|e^*-\sum\limits_{j=1}^{\infty}e^*(n_j)e^*_{n_j}\right\|}
=\frac{e^*-\sum\limits_{j=1}^{\infty}e^*(n_j)e^*_{n_j}}{\left\|e^*\right\|-\sum\limits_{j=1}^{\infty}\left|e^*(n_j)\right|}.$$

\textbf{Step 2.} We claim that $B_Z=D \cap Z$ has the following
property:

\begin{equation}
B_Z=\textrm{conv}(K\cup -K). \tag{$\heartsuit$}
\end{equation}

Indeed, from definition of $K$ and $Z$, it follows that
\begin{eqnarray*}
B_Z&=&\textrm{conv}(B_{\ell_1}\cup K\cup -K)\cap
Z=\textrm{conv}(B_{\ell_1}\cup \textrm{conv}(K\cup -K))\cap Z
\\&=&\textrm{conv}((B_{\ell_1}\cap Z)\cup \textrm{conv}(K\cup -K))
\end{eqnarray*}
so, in order to prove property $(\heartsuit)$, we have to show that

\begin{equation}\label{prop 1}
B_{\ell_1}\cap Z \subset \textrm{conv}(K\cup -K).
\end{equation}

It is easy to see that $\ell_1$ and $Z$ are isometrically isomorphic
via $\phi: \ell_1\rightarrow Z$ defined by $\phi (e_1)=e_0^*$,
$\phi(e_j)=e_{n_{j-1}}^*$, $j\geq 2$, where $(e_j)_{j\geq 1}$
denotes the standard basis in $\ell_1$. Since
$B_{\ell_1}=\overline{\textrm{conv}\left( \left\{\pm e_1,\pm e_2,\pm
e_3,\dots\right\}\right) }$, it follows that

$$B_{\ell_1}\cap Z=\overline{\textrm{conv}\left( \left\{\pm e_0^*,\pm e_{n_1}^*,\pm
e_{n_2}^*,\dots\right\}\right)} .$$ Consequently, in order to prove
(\ref{prop 1}), it is enough to show that
$$\left\{\pm e_0^*,\pm e_{n_1}^*,\pm
e_{n_2}^*,\dots\right\} \subset \textrm{conv}(K\cup -K).$$ Obviously
\begin{equation}\label{prop2}
\left\{\pm \frac{e_{n_1}^*}{(1-\varepsilon)\left\|e^*\right\|},\pm
\frac{e_{n_2}^*}{(1-\varepsilon)\left\|e^*\right\|},\dots\right\}\subset
\textrm{conv}(K\cup -K)
\end{equation}
so $$\left\{\pm e_{n_1}^*,\pm e_{n_2}^*,\dots\right\} \subset
\textrm{conv}(K\cup -K).$$ Therefore it remains to prove that $$\pm
e_0^*\in \textrm{conv}(K\cup -K).$$ The case
$\sum_{j=1}^{\infty}\left|e^*(n_j)\right|=0$ is trivial. Suppose
that $\sum_{j=1}^{\infty}\left|e^*(n_j)\right|\neq 0$. From
(\ref{prop2}) we have
\begin{eqnarray*}
\textrm{conv}(K\cup -K)&\supset&\overline{\textrm{conv}\left( \left\{\pm
\frac{e_{n_1}^*}{(1-\varepsilon)\left\|e^*\right\|},\pm
\frac{e_{n_2}^*}{(1-\varepsilon)\left\|e^*\right\|},\dots\right\}\right) }
\\&=& \frac{1}{(1-\varepsilon)\left\|e^*\right\|}B_{\ell_1} \cap
\overline{\textrm{span}\left( \left\{e_{n_1}^*,e_{n_2}^*,\dots\right\}\right) }
\end{eqnarray*}
and, consequently,
\begin{equation*}
\pm\frac{1}{(1-\varepsilon)\left\|e^*\right\|}\cdot
\frac{\sum\limits_{j=1}^{\infty}e^*(n_j)e^*_{n_j}}{\sum\limits_{j=1}^{\infty}\left|e^*(n_j)\right|}
\in \textrm{conv}(K\cup -K).
\end{equation*}
Therefore, we easily see that for
$t:=\left(\sum_{j=1}^{\infty}\left|e^*(n_j)\right|\right)\cdot
\left(1+\sum_{j=1}^{\infty}\left|e^*(n_j)\right|\right)^{-1}$,

\begin{eqnarray*}
\textrm{conv}(K\cup -K)&\ni& t\cdot
\frac{-\sum\limits_{j=1}^{\infty}e^*(n_j)e^*_{n_j}}
{(1-\varepsilon)\left\|e^*\right\|\sum\limits_{j=1}^{\infty}\left|e^*(n_j)\right|}+
(1-t)\cdot \frac{e^*}{(1-\varepsilon)\left\|e^*\right\|}
\\&=& \frac{e^*-\sum\limits_{j=1}^{\infty}e^*(n_j)e^*_{n_j}}{(1-\varepsilon)\left\|e^*\right\|
(1+\sum\limits_{j=1}^{\infty}\left|e^*(n_j)\right|)} \\&=&
\frac{\left\|e^*\right\|-\sum\limits_{j=1}^{\infty}\left|e^*(n_j)\right|}{(1-\varepsilon)\left\|e^*\right\|
(1+\sum\limits_{j=1}^{\infty}\left|e^*(n_j)\right|)}e_0^*
\end{eqnarray*}
and, since $\textrm{conv}(K\cup -K)$ is symmetric, we get
\begin{equation*}
\pm
\frac{\left\|e^*\right\|-\sum\limits_{j=1}^{\infty}\left|e^*(n_j)\right|}{(1-\varepsilon)\left\|e^*\right\|
(1+\sum\limits_{j=1}^{\infty}\left|e^*(n_j)\right|)}e_0^* \in
\textrm{conv}(K\cup -K).
\end{equation*}
From (\ref{estimate 5}) it follows that
$$\frac{\left\|e^*\right\|-\sum\limits_{j=1}^{\infty}\left|e^*(n_j)\right|}{(1-\varepsilon)\left\|e^*\right\|
(1+\sum\limits_{j=1}^{\infty}\left|e^*(n_j)\right|)}>
\frac{1-\frac{\varepsilon}{4}}{(1-\varepsilon)(1+\frac{\varepsilon}{4})}>1.$$
Thus $\pm e_0^* \in \textrm{conv}(K\cup -K)$ which finishes the
proof of property $(\heartsuit)$.

\textbf{Step 3.} (\textit{Fixed point free nonexpansive map.}) We now
define the operator $T:Z\rightarrow Z$ by
$$T\left(t_0e^*+\sum_{j=1}^{\infty}t_j e_{n_j}^*\right)=\sum_{j=1}^{\infty}t_{j-1} e_{n_j}^* \textrm{ ,} \quad
\sum_{j=0}^{\infty}\left|t_{j}\right| <\infty.$$

It is easy to see (via (\ref{estimate 5})) that the mapping $T$ is
well-defined and linear. We claim that $T(B_Z)\subset B_Z$. Indeed,
property $(\heartsuit)$ ensures that every $x\in B_Z$ has the form
$x=(1-\lambda)y+\lambda z$ for some $\lambda \in [0,1]$, $y \in K$
and $z \in -K$. Therefore, since $T(K)\subset K$ and $T(-K)\subset
-K$, we see that
$$Tx=T((1-\lambda)y+\lambda z)=(1-\lambda)Ty+\lambda Tz \in B_Z.$$
Consequently, since $K \subset S_Z$, $\left\|T\right\|=1$. Clearly,
the restriction $T|_{K}$ of operator $T$ to the $w^*$-compact
convex set $K \subset Y^*$ is a fixed point free nonexpansive map.

Finally, by taking $e^*$ arbitrarily close to the unit sphere
$S_{\ell_1}$ and $\varepsilon\searrow 0$ in (\ref{estimate 6}), we
get a contradiction.

Now, we proceed to prove the implication (ii) $\Rightarrow$ (i).  This part of the proof is an appropriate adaptation of the proof of the main result in \cite{Soardi}.
Let $Y$ be a Banach space isomorphic to $X$ and $A$ be any
isomorphism from $X$ onto $Y$. Let $T$:$K \rightarrow K$ be a nonexpansive mapping, where $K\subset Y^*$ is a convex $w^*$-compact set. It is well known that there exists a sequence $\left\{ x^*_{0,n} \right\} \subset K$ such that 
$\lim_{n \rightarrow \infty} \left\|x^*_{0,n}-Tx^*_{0,n}\right\|_{Y^*}=0$ (see, e.g, \cite{Goebel-Kirk}). We may assume that $\left\{x^*_{0,n}\right\}$ is $w^*$-convergent to $x^*_0 \in K$ and that the limit $\alpha_0=\lim_{n \rightarrow \infty} \left\|x^*_{0,n}-x^*_0 \right\|_{Y^*}$ exists. Since $T$ is a nonexpansive mapping, for every $k \in \mathbb{N}$ we have
$$
\left\|x^*_0-T^kx^*_0\right\|_{Y^*}\leq \limsup_{n\rightarrow \infty}\left\|x^*_{0,n}-T^kx^*_0\right\|_{Y^*}\leq \alpha_0.
$$
Therefore, by (a) in Lemma \ref{lemma Soardi}, there exists a closed convex invariant set $H(x^*_0)\subset K$ such that $\mathrm{diam}(H(x^*_0))\leq \alpha_0$. Then there exists a sequence $\left\{ x^*_{1,n}\right\}\subset H(x^*_0)$ such that
\begin{enumerate}
\item [(1)] $\lim_{n \rightarrow \infty} \left\|x^*_{1,n}-Tx^*_{1,n}\right\|_{Y^*}=0$,
\item [(2)] $\left\{ x^*_{1,n}\right\}$ is $w^*$-convergent to $x_1^*\in K$,
\item [(3)] $\alpha_1=\lim_{n \rightarrow \infty} \left\|x^*_{1,n}-x^*_1 \right\|_{Y^*}$ exists.
\end{enumerate}
Now, we have that $\left\{ A^*x^*_{1,n} \right\}\subset\ell_1$ is $\sigma(\ell_1,X)$-convergent to $A^*x^*_1$. Moreover, there is no loss of generality in assuming that $\lim_{n \rightarrow \infty} \left\|A^*x^*_{1,n}-A^*x^*_1\right\|_{\ell_1}$ exists. Hence, by (c) in Lemma \ref{Lemma sequences} we obtain
\begin{eqnarray*}
\alpha_0 &\geq& \limsup_{m \rightarrow \infty}\limsup_{n \rightarrow \infty}\left\|x^*_{1,m}-x^*_{1,n}\right\|_{Y^*}
\\ &\geq& \dfrac{1}{\|A\|}\limsup_{m \rightarrow \infty}\limsup_{n \rightarrow \infty}\left\|A^*x^*_{1,m}-A^*x^*_{1,n}\right\|_{\ell_1}
\\ &\geq& \dfrac{1}{\|A\|} \dfrac{2}{1+r}\lim_{n \rightarrow \infty}\left\|A^*x^*_{1,n}-A^*x^*_{1}\right\|_{\ell_1}
\\ &\geq& \dfrac{1}{\|A^{-1}\|\,\|A\|} \dfrac{2}{1+r}\lim_{n \rightarrow \infty}\left\|x^*_{1,n}-x^*_{1}\right\|_{Y^*}=\dfrac{1}{\|A^{-1}\|\,\|A\|} \dfrac{2}{1+r}\alpha_1.
\end{eqnarray*}
From the inequalities above, we conclude that $\alpha_1 \leq \|A^{-1}\|\,\|A\| \dfrac{1+r}{2}\alpha_0$.
Moreover, since $\{x^*_0-x^*_{1,n}\}$ is $w^*$-convergent to $x^*_0-x^*_1$, (b) in Lemma \ref{lemma Soardi} yields $\|x^*_0-x^*_1\|_{Y^*}\leq 2\alpha_0$.

Repeated applications of this construction give us a sequence $\{\alpha_n \}$ of non negative numbers such that
\begin{equation}\label{sequence alpha}
\alpha_{n+1}\leq \left[\|A^{-1}\|\,\|A\| \dfrac{1+r}{2}\right]^{n+1}\alpha_0
\end{equation}
and a sequence $\{x^*_n\}\subset Y^*$ such that $\|x^*_n-x^*_{n+1}\|_{Y^*}\leq 2\alpha_n$ and $\left\|x^*_n-Tx^*_n\right\|_{Y^*}\leq \alpha_n$.
From inequality (\ref{sequence alpha}), we see at once that the sequence $\{\alpha_n\}$ converges to $0$ if  $\|A^{-1}\|\,\|A\|< \frac{2}{1+r}$. Moreover, if this condition holds, the sequence $\{x^*_n\}$ strongly converges to a fixed point of $T$. This fact concludes the proof by showing that $Y^*$ has the $w^*$-fpp whenever $d(X,Y)< \frac{2}{1+r}$.
\end{proof}

\begin{rem}\label{remark Example Lim}
It is easy to observe that, from the proof of Theorem \ref{Theorem stable fpp}, we obtain also a quantitative estimation of the stability constant $\gamma=\frac{2}{1+r}$.
Moreover, it is worth pointing out that the estimation above is sharp when $r=0$ as shown by the example contained in \cite{Lim1980}. On the other hand the proof of the sharpness of the present estimation when $0 < r<1$ remains as an open problem.   
\end{rem}

\section{Polyhedrality in Lindenstrauss spaces}\label{sec:poly}
The aim of this section is to show that the geometrical properties used in Theorems \ref{Theorem fpp} and \ref{Theorem stable fpp} are strictly related to polyhedrality.
The starting point to consider polyhedrality in an infinite-dimensional setting is the definition given by Klee in \cite{Klee1960}, where he extended the notion of convex finite-dimensional polytope for the case of the closed unit ball $B_X$ of an infinite-dimensional Banach space $X$. Thenceforth polyhedrality was extensively studied and several different definitions have been stated. For a detailed account about these definitions and their relationships see \cite{Durier-Papini1993} and \cite{Fonf-Vesely2004}. Here we restrict our attention to some of the definitions collected in \cite{Fonf-Vesely2004}. Moreover, we show that the property introduced in Theorem \ref{Theorem fpp} can be considered as a new definition of polyhedrality (namely, property (pol-iii)) since we will prove that this property is placed in an intermediate position between other polyhedrality notions that are already considered in the literature. 
\begin{dfn}\label{Dfn:Polyhedrality}
Let $X$ be a Banach space. We consider the following properties of $X$:
\begin{enumerate}
 \item [(pol - i)] $\left(\mathrm{ext} \, B_{X^*}\right)'\subset \left\lbrace 0\right\rbrace$ (\cite{Maserick1965});
 \item [(pol - ii)] $\left(\mathrm{ext} \, B_{X^*}\right)'\subset rB_{X^*}$ for some $0<r<1$ (\cite{Fonf-Vesely2004});
 \item [(pol - iii)] there is no infinite set $C \subset \mathrm{ext} \, B_{X^*}$ such that $\overline{\mathrm{conv} (C)}^* \subset S_{X^*}$;
 \item [(pol - iv)] there is no infinite-dimensional $w^*$-closed proper face of $B_{X^*}$ (\cite{Lazar1969});
 \item [(pol - v)] $x^*(x)<1$ whenever $x \in S_{X}$ and $x^* \in \left(\mathrm{ext} \, B_{X^*}\right)'$ (\cite{Gleit-McGuigan1972});
 \item [(pol - vi)] the set $\mathrm{ext} \, D(x)$ is finite for each $x \in S_X$ (property ($\Delta$) in \cite{Fonf-Vesely2004});
 \item [(pol - vii)] $\sup \left\lbrace x^*(x) : x^*\in \mathrm{ext} \, B_{X^*} \setminus D(x)\right\rbrace <1$ for each $x \in S_X$ (\cite{BrosowskiDeutsch1974});
 \item[(pol-K)] the unit ball of every finite-dimensional subspace of $X$ is a polytope (\cite{Klee1960}).
\end{enumerate}
\end{dfn}

The following theorem clarifies the relationships between the various notions of polyhedrality stated in Definition \ref{Dfn:Polyhedrality} in the framework of Lindenstrauss spaces.
\begin{thm}\label{Thm:relationships polyhedrality}
Let $X$ be a Lindenstrauss space. The following relationships hold:
\[
\mathrm{(pol-i)}\Rightarrow\mathrm{(pol-ii)}\Rightarrow\mathrm{(pol-iii)}\Rightarrow\begin{array}{c}
\mathrm{(pol-v)}\\
\Updownarrow\\
\mathrm{(pol-iv)}\\
\Updownarrow\\
\mathrm{(pol-vi)}
\end{array}\Rightarrow\mathrm{(pol-vii)}\mathrm{\Leftrightarrow(pol-K)}.
\]
\end{thm}
\begin{proof}
The implications
\[
\begin{array}{c}
\mathrm{(pol-i)}\Rightarrow\mathrm{(pol-ii)}\Rightarrow\mathrm{(pol-iii)}\Rightarrow\mathrm{(pol-iv)},\\
\mathrm{(pol-v)}\Rightarrow\mathrm{(pol-vi)}
\end{array}
\]
are trivial.
The proof of Theorem 1.2 (p. 402-403) in \cite{Gleit-McGuigan1972} shows that $$\mathrm{(pol-iv)}\Rightarrow\mathrm{(pol-v)}.$$
The implication
$$\mathrm{(pol-vi)}\Rightarrow\mathrm{(pol-iv)}$$
follows easily from Lemma \ref{Lem:weak^*-closed faces} below. The implication
$$
\mathrm{(pol-v)}\Rightarrow\mathrm{(pol-vii)}
$$
is proved in Theorem 1 in \cite{Durier-Papini1993}. Finally the equivalence
$$\mathrm{(pol-vii)}\Leftrightarrow\mathrm{(pol-K)}$$
is proved in Theorem 4.3 in \cite{Casini-Miglierina-Piasecki-Vesely2016}.
\end{proof}
The previous proof needs the following lemma, that is interesting in itself as it gives a property of the $w^*$-closed faces of $B_{X^*}$. Indeed, this property is strictly related to the compact norm-preserving extension of compact operators (see \cite{Casini-Miglierina-Piasecki-Vesely2016}).
\begin{lem}\label{Lem:weak^*-closed faces}
Let $X$ be a Lindenstrauss space and let $F$ be a $w^*$-closed proper face of $B_{X^*}$. Then there exists $x\in S_X$ such that $F\subseteq D(x)$.
\end{lem}
\begin{proof}
Let us fix an element $\bar{x}^*\in F$. Then we consider the subspace $V=\mathrm{span}(\bar{x}^*-F)$. It is easy to prove that for every $h^*\in H=\mathrm{conv}\left( F\cup(-F)\right)$ there exists a unique real number $\alpha \in [-1,1]$ such that
\[
h^* \in \alpha \bar{x}^*+V.
\]
Now, let $A_0(H)$ denotes the Banach space of all $w^*$-continuous affine symmetric (i.e., $f(-x)=-f(x)$) functions on $H$.  We introduce the function $\bar{y}:H\rightarrow \mathbb{R}$ defined by $\bar{y}(h^*)=\alpha$. It is easy to recognize that $\bar{y}\in A_0(H)$ and that $\bar{y}|_F=1$. By considering the separable subspace $Y=\mathrm{span}\left( \{\bar{y}\}\right) $ of $A_0(H)$, we can apply Proposition 1 in \cite{Lazar1969} to show that there exists an isometry $T:Y\rightarrow X$ such that $h^*(T(y))=y(h^*)$ for every $h^*\in H$ and for every $y \in Y$. Therefore $T(\bar{y})\in S_X$ and $F\subseteq D(T(\bar{y}))$.
\end{proof}

The equivalence among (pol-iv), (pol-v) and (pol-vi) allows us to clarify the situation about the existence of compact norm-preserving extension of a compact operator with values in a Lindenstrauss space.
From  \cite{Lindenstrauss 1964} it is known that a Lindenstrauss space $X$ is polyhedral if and only if for every Banach spaces $Y\subset Z$ and every operator $T:Y\rightarrow X$ with $\dim T(Y)\leq2$ there exists a compact extension $\tilde{T}:Z\rightarrow X$ with $\left\Vert \tilde{T}\right\Vert =\left\Vert T\right\Vert $ (combine Theorem 7.9 in \cite{Lindenstrauss 1964} and Theorem 4.7 in \cite{Klee 1959}). Moreover, Lazar provided (see Theorem 3 in \cite{Lazar1969})  the following more general extension property that he asserted to be equivalent to polyhedrality.

\begin{fact} \label{thm:Lazar}
\rm If $X$ is a
Lindenstrauss space, then the following properties are equivalent:

\begin{itemize}
\item [\rm {(1)}] $X$ is a polyhedral space;
\item [\rm {(2)}] $X$ does not contain an isometric copy of $c$;
\item [\rm {(3)}] there are no infinite-dimensional $w^{*}$- closed
proper faces of $B_{X^{*}}$;
\item [\rm{(4)}] for every Banach spaces $Y\subset Z$ and every
compact operator $T:Y\rightarrow X$ there exists a compact extension
$\tilde{T}:Z\rightarrow X$ with $\left\Vert \tilde{T}\right\Vert=
\left\Vert T\right\Vert $.
\end{itemize}
\end{fact}
However, in \cite{Casini-Miglierina-Piasecki-Vesely2016} it is shown that the equivalences between (1) and (4) and between (1) and (3) in Fact \ref{thm:Lazar} are false. Therefore, some of the considered implications remain unproven. 
On the other hand, a characterization of norm-preserving extendability of a compact operators is provided. Indeed, the following result holds.
\begin{thm}[Theorem 5.3 in \cite{Casini-Miglierina-Piasecki-Vesely2016}]
For an infinite-dimensional Banach space $X$, the following assertions
are equivalent.
\begin{enumerate}
\item $X$ is a Lindenstrauss space such that each $D(x)$ 
($x\in S_X$) is finite-dimensional.
\item For every Banach spaces $Y\subset Z$, every compact operator
$T\colon Y\to X$ admits a compact norm-preserving extension
$\tilde{T}\colon Z\to X$.
\end{enumerate}
\end{thm}

Since the set $D(x)$ is finite-dimensional if and only if $\mathrm{ext}\, D(x)$ is finite (see Remark 5.2 in \cite{Casini-Miglierina-Piasecki-Vesely2016}), Theorem \ref{Thm:relationships polyhedrality} gives a correct version of the result of Lazar quoted in Fact \ref{thm:Lazar}.

\begin{thm}\label{thm:compact-polyhedrality}
Let $X$ be an infinite-dimensional Banach space. The following assertions are equivalent:
\begin{enumerate}
\item $X$ is a Lindenstrauss space enjoying property (pol-iv);
\item $X$ is a Lindenstrauss space enjoying property (pol-v);
\item $X$ is a Lindenstrauss space enjoying property (pol-vi);
\item For every Banach spaces $Y\subset Z$, every compact operator
$T\colon Y\to X$ admits a compact norm-preserving extension
$\tilde{T}\colon Z\to X$.
\end{enumerate}
\end{thm} 

It is worth to pointing out that the original statement of Lazar (see Fact \ref{thm:Lazar}) has a correct version where two separate groups of equivalent properties are recognized. Namely, Theorem 4.3 in \cite{Casini-Miglierina-Piasecki-Vesely2016} shows that properties (1) and (2) are equivalent, while the equivalence between (3) and (4) follows from Theorem 5.3 in \cite{Casini-Miglierina-Piasecki-Vesely2016} and Theorem \ref{thm:compact-polyhedrality} above.

Now, we prove that none the one-side implications of Theorem \ref{Thm:relationships polyhedrality} can be reversed. In \cite{Durier-Papini1993,Fonf-Vesely2004} there are many examples proving that the considered implications cannot be reversed when a general Banach space is considered, but most of them are not Lindenstrauss spaces. On the other hand, the following examples show that the implications cannot be reversed even if we restrict our attention to the class of Lindenstrauss spaces. All of them are based on suitable hyperplanes of the space $c$ of the convergent sequences.
Let $\alpha=(\alpha(1),\alpha(2),\ldots)\in B_{\ell_1}$, we define the space
$$
W_\alpha=\left\lbrace x=(x(1),x(2),\ldots)\in c: \lim_{i \rightarrow \infty}x(i)=\sum_{i=1}^{+\infty}\alpha(i)x(i) \right\rbrace.
$$ 
A detailed study of this class of spaces was developed in \cite{Casini-Miglierina-Piasecki2014} and in Section 2  of \cite{Casini-Miglierina-Piasecki2015}. Here we recall only that $W_\alpha$ is a predual of $\ell_1$ and that the standard basis $\left\lbrace e^*_n\right\rbrace$ of $\ell_1$ is $\sigma(\ell_1,W_\alpha)$-convergent to $\alpha$ for every $\alpha \in B_{\ell_1}$.

\begin{exa}\label{ex1}[$\mathrm{(pol-ii)}\nRightarrow\mathrm{(pol-i)}$]
Let $\alpha=(\frac{r}{2},\frac{r}{2},0,0,\ldots)\in \ell_1$ for $0<r<1$. Then 
the standard basis $\left\lbrace e^*_n\right\rbrace$ is $\sigma(\ell_1,W_\alpha)$-convergent to 
$\alpha$.
\end{exa}

\begin{exa}\label{ex2}[$\mathrm{(pol-iii)}\nRightarrow\mathrm{(pol-ii)}$]
Let $\alpha=(-\frac{1}{2},-\frac{1}{4},-\frac{1}{8},\ldots)\in \ell_1$. It is easy to see that $W_\alpha$ satisfies (pol-iii), but the standard basis $\left\lbrace e^*_n\right\rbrace$ is $\sigma(\ell_1,W_\alpha)$-convergent to $\alpha$.
\end{exa}

\begin{exa}\label{ex3}[$\mathrm{(pol-iv)}\nRightarrow\mathrm{(pol-iii)}$]
Let $\alpha=(\frac{1}{2},-\frac{1}{4},\frac{1}{8},-\frac{1}{16},\ldots)\in \ell_1$.
By considering the set $C=\left\lbrace e^*_1,e^*_3,e^*_5,\ldots \right\rbrace$, it is easy to recognize that $W_\alpha$ fails property $\mathrm{(pol-iii)}$. However $W_\alpha$ satisfies (pol-iv).
\end{exa}

\begin{exa}\label{ex4}[$\mathrm{(pol-vii)}\nRightarrow\mathrm{(pol-iv)}$]
Let $\alpha=(\frac{1}{2},\frac{1}{4},\frac{1}{8},\ldots)\in \ell_1$. 
A detailed study of the properties of $W_\alpha$ is carried out in \cite{Casini-Miglierina-Piasecki-Vesely2016}. Here, we recall only that there we proved that $W_\alpha$ satisfies $\mathrm{(pol-vii)}$ (and hence it enjoys also $\mathrm{(pol-K)})$ but it lacks $\mathrm{(pol-iv)}$.
\end{exa}




\begin{thebibliography}{1}


\bibitem{BrosowskiDeutsch1974}B. Brosowski, F. Deutsch. On some geometric properties of suns. \textit{J. Approx. Theory} \textbf{10} (1974), 245-267. 


\bibitem{Casini-Miglierina-Piasecki2014} E. Casini, E. Miglierina,
\L. Piasecki. Hyperplanes in the space of convergent sequences and
preduals of $\ell_1$. \emph{Canad. Math. Bull.} \textbf{58} (2015), 459-470.

\bibitem{Casini-Miglierina-Piasecki2015} E. Casini, E. Miglierina, \L. Piasecki. Separable Lindenstrauss spaces whose duals lack the weak$^*$ fixed point property for nonexpansive mappings. To appear in \emph{Studia Math.}

\bibitem{Casini-Miglierina-Piasecki-Vesely2016} E. Casini, E. Miglierina,
\L. Piasecki and L. Vesel\'{y}. Rethinking polyhedrality for Lindenstrauss spaces. \emph{Israel J. Math.} \textbf{216} (2016), 355-369.

\bibitem{Dominguez-Garcia-Japon1998} T. Dom\'{\i}nguez Benavides, J. Garc\'{\i}a Falset, M. A. Jap\'{o}n Pineda. The $\tau$-fixed point property for nonexpansive mappings. \textit{Abstr. Appl. Anal.} \textbf{3} (1998), 343-362.


\bibitem{Durier-Papini1993} R. Durier, P. L. Papini. Polyhedral norms in an infinite dimensional space. \emph{Rocky Mt. J. Math.} \textbf{23} (1993), 863-875.

\bibitem{Fonf-Lindenstrauss-Phelps2001} V.P. Fonf, J. Lindenstrauss, R.R. Phelps. Infinite dimensional convexity. In Handbook of the Geometry of Banach Spaces 1 (eds.: W.B. Johnson and J. Lindenstrauss), Elsevier Science, 2001, 599-670.

\bibitem{Fonf-Vesely2004} V. P. Fonf, L. Vesel\'{y}. Infinite-dimensional polyhedrality. \emph{Canad. J. Math.} \textbf{56} (2004), 472-494.


\bibitem{Gleit-McGuigan1972} A. Gleit, R. McGuigan. A note on
polyhedral Banach spaces. \textit{Proc. Amer. Math. Soc.} \textbf{33} (1972), 398-404.

\bibitem{Goebel-Kirk} K. Goebel, W.A. Kirk. Topics in metric fixed point theory. Cambridge Univ. Press. 1990.


\bibitem{Karlovitz1976} L. A. Karlovitz. On nonexpansive mappings. \textit{Proc. Amer. Math. Soc.} \textbf{55}
(1976), 321-325.


\bibitem{Klee 1959} V. Klee. Some characterizations of convex polyhedra, \emph{Acta Math.} \textbf{102} (1959), 79-107.

 

\bibitem{Klee1960} V. Klee. Polyhedral sections of convex bodies. \textit{Acta Math.} \textbf{103} (1960), 243-267.

%

\bibitem{Lazar1969} A. J. Lazar. Polyhedral Banach spaces and
extensions of compact operators. \textit{Israel J. Math.} \textbf{7}
(1969), 357-364.

\bibitem{Japon-Prus2004}M. A. Jap\'on-Pineda, S. Prus. Fixed point
property for general topologies in some Banach spaces. \emph{Bull.
Austral. Math. Soc.} \textbf{70} (2004), 229-244.

\bibitem{Lazar-Lindenstrauss1971}A. J. Lazar, J. Lindenstrauss.
Banach spaces whose duals are $L_1$-spaces and their representing
matrices. \textit{Acta Math.} \textbf{126} (1971), 165-194.

\bibitem{Lindenstrauss 1964}J. Lindenstrauss. Extension of compact operators, \emph{Mem. Amer. Math. Soc.} \textbf{48} (1964).

\bibitem{Lim1980}T. C. Lim. Asymptotic centers and nonexpansive mappings in conjugate Banach spaces. \textit{Pacific J. Math.} \textbf{90} (1980), 135-143.

\bibitem{Maserick1965}P. H. Maserick. Convex polytopes in linear spaces. \textit{Illinois J. Math.} \textbf{9} (1965), 623-635.




\bibitem{Soardi} P. M. Soardi. Schauder basis and fixed points of nonexpansive mappings. \emph{Pacific. J. Math.} \textbf{101} (1982), 193-198.


\end{thebibliography}

\end{document}